\documentclass[a4paper,12pt]{amsart}
\usepackage{amssymb,a4wide,ifpdf}

\ifpdf
  \usepackage[pdftex]{graphicx,color}
  \usepackage[pdftex,colorlinks]{hyperref}
  \hypersetup{%
      pdftitle={Sampling and Interpolation on Finite Riemann Surfaces},
      pdfauthor={Joaquim Ortega-Cerda},
      colorlinks=true,
      linkcolor=blue 
  }
\else
  \usepackage[dvips]{graphicx,color}
  \usepackage[hypertex,colorlinks]{hyperref}
  \hypersetup{%
      colorlinks=true,
      linkcolor=blue 
  }
\fi

\newcommand{\D}{\mathbb D}
\newcommand{\C}{\mathbb C}
\newcommand{\Z}{\mathbb Z}
\renewcommand{\H}{\mathcal H}
\newcommand{\ddbar}{\partial\bar\partial}
\newcommand{\dbar}{\bar\partial}

\renewcommand{\Re}{\operatorname{Re}}

\newtheorem{theorem}{Theorem}
\newtheorem{proposition}[theorem]{Proposition}
\newtheorem{lemma}[theorem]{Lemma}
\theoremstyle{definition}
\newtheorem{definition}[theorem]{Definition}

\title{Interpolating and sampling sequences in finite Riemann surfaces}
\author{Joaquim Ortega-Cerd\`a}
\address{Dept.\ Matem\`atica Aplicada i An\`alisi, Universitat  de Barcelona,
Gran Via 585, 08071 Bar\-ce\-lo\-na, Spain}
\email{jortega@ub.edu}
\date{Working draft: \today}
\thanks{Supported by DGICYT grant MTM2005-08984-C02-02 and the CIRIT grant
2005SGR00611}


\begin{document}
\maketitle
\begin{abstract}
We provide a description of the interpolating and sampling sequences on 
a space of holomorphic functions with a uniform growth restriction defined on 
finite Riemann surfaces. 
\end{abstract}

\section{Introduction and statement of the results}
Let $S$ be an open finite Riemann surface endowed with the Poincar\'e
(hyperbolic) metric. We will study some properties of holomorphic functions in
the Riemann surface with uniform growth control. Namely we will deal with the
Banach space $A_\phi(S)$ of holomorphic functions in $S$ such that
$\|f\|:=\sup_S |f|e^{-\phi}<\infty$ where $\phi$ is  a given subharmonic
function that controls the growth of the functions in the space.

The fact that $\phi$ is subharmonic is a natural assumption on the weight that
limits the growth since any other growth control given by a weight $\psi$,
$\|f\|_*=\sup_S |f|e^{-\psi}$ can be replaced by an equivalent subharmonic
function because $\phi=\sup_{\|f\|_*\le 1}\log|f|$ is a subharmonic function
and $A_\psi(S)=A_\phi(S)$ with equality of norms, $\sup_S |f|e^{-\psi}= \sup_S
|f|e^{-\phi}$.

We have fixed a metric. It is then natural to restrict the possible weights
$\phi$, in a way that the functions in $A_\phi$ oscillate in a controlled way
when the points are nearby in the Poincar\'e metric. This is achieved for
instance by assuming that $\phi$ has  bounded Laplacian (the Laplace-Beltrami
operator with respect to the hyperbolic measure). That is, if in a local
coordinate chart the Poincar\'e metric is of the form $ds^2= e^{2\nu(z)}|dz|^2$,
then we assume that $\Delta\phi = 4 e^{-2\nu(z)} \frac{\partial^2\phi}{\partial
z\partial \bar z}$ satisfies $C^{-1}\le \Delta \phi \le C$. If we want to deal
with other weights then it is possible to introduce a natural metric associated
to the weight as it is done in the plane in \cite{MarMasOrt03}. In this work we
will only consider the Poincar\'e metric and bounded Laplacian since it already
covers many interesting cases and it is technically simpler.

The problems that we will consider are the following:
\begin{enumerate}
\item[(A)] The description of the interpolating sequences for $A_\phi(S)$: i.e.
the sequences $\Lambda\subset S$ such that it is always possible to find an
$f\in A_\phi(S)$ such that $f(\lambda)=v_\lambda$ for all $\lambda\in \Lambda$
whenever the data $\{v_\lambda\}_{\Lambda}$, satisfies  the compatibility
condition $\sup_\Lambda |v_\lambda|e^{-\phi(\lambda)}<+\infty$ 
\item[(B)] The description of sampling sets for $A_\phi(S)$: i.e. the sets
$E\subset S$ such that there is a constant $C>0$ that satisfies
\[
\sup_{S} |f|e^{-\phi}\le C \sup_{E} |f|e^{-\phi},\quad \forall f\in A_\phi(S).
\]
\end{enumerate}

In the solution of these problems the Poincar\'e distance  and the potential
theory in the surface play a key role. This has already been observed by
A.~Schuster and D.~Varolin in \cite{SchVar04}, where they provide sufficient
conditions for a sequence to be interpolating/sampling for functions in a
slightly different context where the weighted uniform control of the growth of
the functions is replaced by a weighted $L^2$ control. Their condition basically
coincides with the description that we reach so our work can be considered as
the counterpart of their theorems, although we will give a different proof of
their results as well. We will rely on the well-known case of the disk and some
simplifying properties of finite Riemann surfaces. Their method  of proof looks
more promising if one wants to extend the result to Riemann surfaces with more
complicated topology.

When the surface is a disk, which will be our model situation, the corresponding
problems have been solved in \cite{BerOrt95}, \cite{OrtSei98} and in a different
way in \cite{Seip98}. Of course, the more basic problem of describing the
interpolating sequences for bounded holomorphic functions in finite Riemann
surfaces (in our notation $\phi\equiv 0$), has been known for a long time, see
\cite{Stout65}).

We introduce now some definitions that will be needed to state our results. For
any point $z\in S$ and any $r>0$ we denote by $D(z,r)$ the domain in the surface
$S$ that consits of points at hyperbolic distance from $z$ less than $r$. They
are topological disks if the center $z$ is outside a big compact of $S$, or if
$r$ is small enough, as we will see in Section~\ref{RS}.

A sequence $\Lambda$ of points in $S$ is hyperbolically separated if  there is
an $\varepsilon>0$ such that the domains
$\{D(\lambda,\varepsilon)\}_{\lambda\in\Lambda}$ are pairwise disjoint.

Let $g_r(z,w)$ be the Green function associated to the surface $D(z,r)$ with
pole at the ``center'' $z$ and $g(z,w)=g_\infty(z,w)$ be the Green function
associated to the surface $S$. We define the densities
\begin{equation}\label{densities}
\begin{split}
D^+_\phi(\Lambda):=\limsup_{r\to\infty} \sup_{z\in S} \frac{\displaystyle
\sum_{\begin{subarray}{c}1/2<d(z,\lambda)<r\\
\lambda\in\Lambda\end{subarray}}g_r(z,\lambda)}
{\displaystyle\int_{D(z,r)}g_r(z,w)i\ddbar\phi(w) }.\\ 
D^-_\phi(\Lambda):=\liminf_{r\to\infty} \inf_{z\in
S}\frac{\displaystyle \sum_{\begin{subarray}{c}1/2<d(z,\lambda)<r\\
\lambda\in\Lambda\end{subarray}}g_r(z,\lambda)}
{\displaystyle\int_{D(z,r)}g_r(z,w)i\ddbar\phi(w) }.
\end{split}
\end{equation}

The main result is
\begin{theorem}\label{main}
Let $S$ be a finite Riemann surface and let $\phi$ be a
subharmonic function with bounded Laplacian. 
\begin{enumerate} 
\item [(A)]A sequence $\Lambda\subset S$ is an interpolating sequence for
$A_\phi(S)$ if and only if it is hyperbolically separated and
$D^+_\phi(\Lambda)<1$. 
\item[(B)] A set $E\subset S$ is a sampling set for $A_\phi(S)$ if and only if
it contains an hyperbolically separated sequence
$\Lambda\subset E$ such that $D^-_\phi(\Lambda)>1$.
\end{enumerate}
\end{theorem}

In Section~\ref{RS} we will prove some key properties of finite Riemann
surfaces. In particular we need to study the behavior of the hyperbolic metric
as we approach the boundary of the surface. We will also prove some weighted
uniform estimates for the inhomogeneous Cauchy-Riemann equation in the surface,
Theorem~\ref{dbar}, that has an interest by itself.

In the next section, we use the tools and Lemmas proved in Section~\ref{RS} to
reduce the interpolating and sampling problem in $S$ to a problem near the
boundary that can be reduced to the known case of the disk.

Finally in Section~\ref{Lebesgue} we show how our results can be extended to
other Banach spaces of holomorphic functions where the uniform growth is
replaced by weighted $L^p$ spaces. 

A final word on notation. By $f\lesssim g$ we mean that there is a constant $C$
independent of the relevant variables such that $f\le C g$ and by $f\simeq g$ we
mean that $f\lesssim g$ and $g\lesssim f$.

\section{Basic properties of finite Riemann surfaces}\label{RS}
We start by the definition and then we collect some properties of $S$ that 
follow from the restrictions 
that we are assuming on the topology of $S$.
\begin{definition}
A finite Riemann Surface is  the interior of a smooth bordered compact Riemann
surface. 
\end{definition}
Our surface is an open Riemann surface and it is in fact an open subset of a
compact surface  (the double, see \cite{SchSpe54}). See Figure~\ref{fsurface}
for a typical representation. Observe that the genus is finite and the border of
the surface consists of a finite number of smooth closed Jordan curves. In most
of what follows the particular case of a smooth finitely connected open set  in
$\C$ has all the difficulties of the general case. 

The following claim follows from instance from \cite[Prop 7.1-7.4]{Scheinberg78}
\begin{lemma}\label{dbaracotat}
For any $(0,1)$-form $\omega$ there is a solution $u$ to the inhomogeneous
Cauchy-Riemann equation $\dbar u =\omega$. Moreover since $S$ has an essential
extension to a compact Riemann surface if the data is a smooth form with compact
support $K$ in $S$ then there is a bounded linear solution $u=T[w]$  with the
bound $|u|\le C_K \langle\omega\rangle$. 
\end{lemma}
In this statement and in the following $\langle \omega\rangle$ is the Poincar\'e
length of the $(0,1)$-form $\omega$.

In the disk we have Blaschke factors that are very convenient to divide out
zeros of holomorphic functions without changing essentially the norm. The
analogous functions that provide us with the same property in the case of finite
Riemann surfaces are given by the next proposition:
\begin{proposition}\label{blashcke} There is a constant $C=C(S)>0$ such that
for any point $z\in S$ there is a function $h_z\in \H(S)$ with
\[
\sup_{w\in S} |\log|h_z(w)|-g(z,w)|<C.
\] 
In particular $h_z(w)$ is a bounded
holomorphic function that vanishes only on the point $z$ and for any
$\varepsilon>0$  $K>|h_z(w)|>C(\varepsilon)$ if $d(z,w)>\varepsilon$.
\end{proposition}
\begin{proof}
The obstruction for an harmonic function $u$ to have an harmonic conjugate is
that for a  set of generators $\{\gamma_i\}_{i=1}^m$
 of the homology we have $\int_{\gamma_i} *d u =0$,  $i=1,\ldots,m$. If we want
$u=\log|f|$ for an $f\in \mathcal H(S)$, we just need that  
$\int_{\gamma_i} *d u \in\mathbb Z$.

Being a finite Riemann surface there are $\{h_j\}_{j=1}^n$ functions in the
algebra of $S$ without zeros such that $\int_{\gamma_i}
*d\log|h_j|=\delta_{ij}$, see \cite[Lemma~1]{Wermer64}. Thus the function 
\[
v(z)=u(z)-\sum_i \Bigl(\int_{\gamma_i} *d u\Bigr) \log|h_i(z)|
\]
is the logarithm of an holomorphic function $\log|f|=v$. Therefore there is a
constant $C$ such that any harmonic function $u$ in $S$ admits an holomorphic
function $f$ with $|u-\log|f||<C$. Take a point $z\in S$ and any holomorphic
function $k_z\in \mathcal H(S)$ that vanishes only on  $z$. Then
$g(z,w)-\log|k_z(w)|$ is harmonic in $S$ and therefore there is a holomorphic
function $f_z$ such that $|g(z,w)-k_z(w)-\log |f_z| |<C$. Thus we may define
$h_z(w)=f_z(w) k_z(w)$ and it has the estimate $|g(z,w)-\log |h_z| |<C$. The
estimate  $|g(z,w)|>C(\varepsilon)$ when $d(z,w)>\varepsilon$ holds in finite
Riemann surfaces, see for instance \cite[Theorem~5.5]{Diller95}. 
\end{proof}

\begin{figure}
\begin{center}
\includegraphics{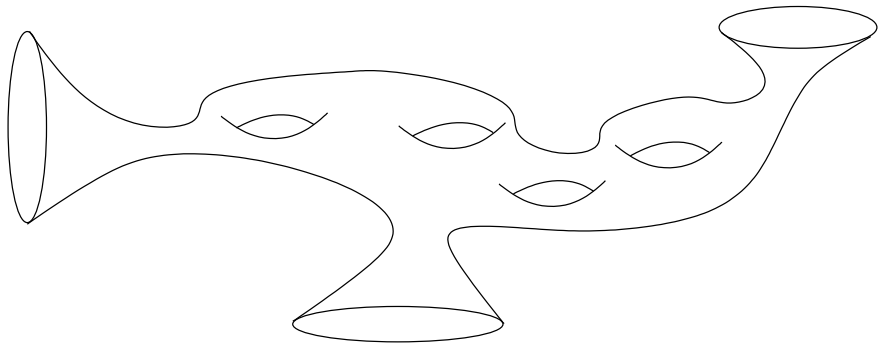}
\end{center}
\caption{A finite Riemann surface with three funnels}\label{fsurface}
\end{figure}

\subsection{The hyperbolic metric in a finite Riemann surface}
The open ends of the Riemann surface can be  parametrized as follows:
The border of the Riemann surface $S$ is a finite union of
smooth closed curves $\tilde\gamma_i$, $i=1,\ldots,n$. Near each
$\tilde\gamma_i$ there is a closed geodesic $\gamma_i$ that is homotopic to
$\tilde\gamma_i$. The subdomain of $S$ bounded by $\gamma_i$ and $\tilde
\gamma_i$ is denoted a ``funnel'' following the terminology of
\cite{DPRS87} and \cite{Diller01}. 

We need to be more precise about the hyperbolic metric in the funnel. There are
nice  coordinates in the funnel that provide good estimates. These are given by
the collar theorem. Let $\D$ be the universal holomorphic cover of $S$ and let
$T_\gamma\in \operatorname{Aut}(\D)$ be the deck transformation corresponding to
the closed loop $\gamma$. Consider the surface $Y=\D/\{T_\gamma^n\}_{n\in\Z}$.
This an annulus since $\pi_1(Y)=\Z$. If we quotient it by the rest of the deck
transformations of the universal cover we get an holomorphic covering map
$\pi_{\gamma}$ from $Y\to S$ which is a local isometry (in $Y$ and $S$ we
consider the Poincar\'e metric inherited from $\D$). In fact
$Y=\{e^{-R}<|z|<e^R\}$, where $R=\pi^2/\operatorname{Length}(\gamma)$,
and $\pi_\gamma$ maps the unit circle isometrically to $\gamma$. Moreover
$\pi_\gamma$ is an isometric injection of the outer part of the annulus
$\{1<|z|<e^{R}\}$ onto the funnel. These will be called the standard coordinates
of the funnel. See \cite{Diller01} and \cite{Buser92} for details.

The Poincare metric in the the funnel is explicit in the standard coordinates
and it is comparable to the hyperbolic metric on the disk in the coordinate
disk $|z|<e^R$ when restricted to $|z|>1$. 

We denote by $A_i$, $i=1,\ldots,n$ the funnels of $S$ bounded by $\gamma_i$ and
$\tilde\gamma_i$.

\begin{figure}\label{funnel}
\begin{center}
\includegraphics[width=10cm]{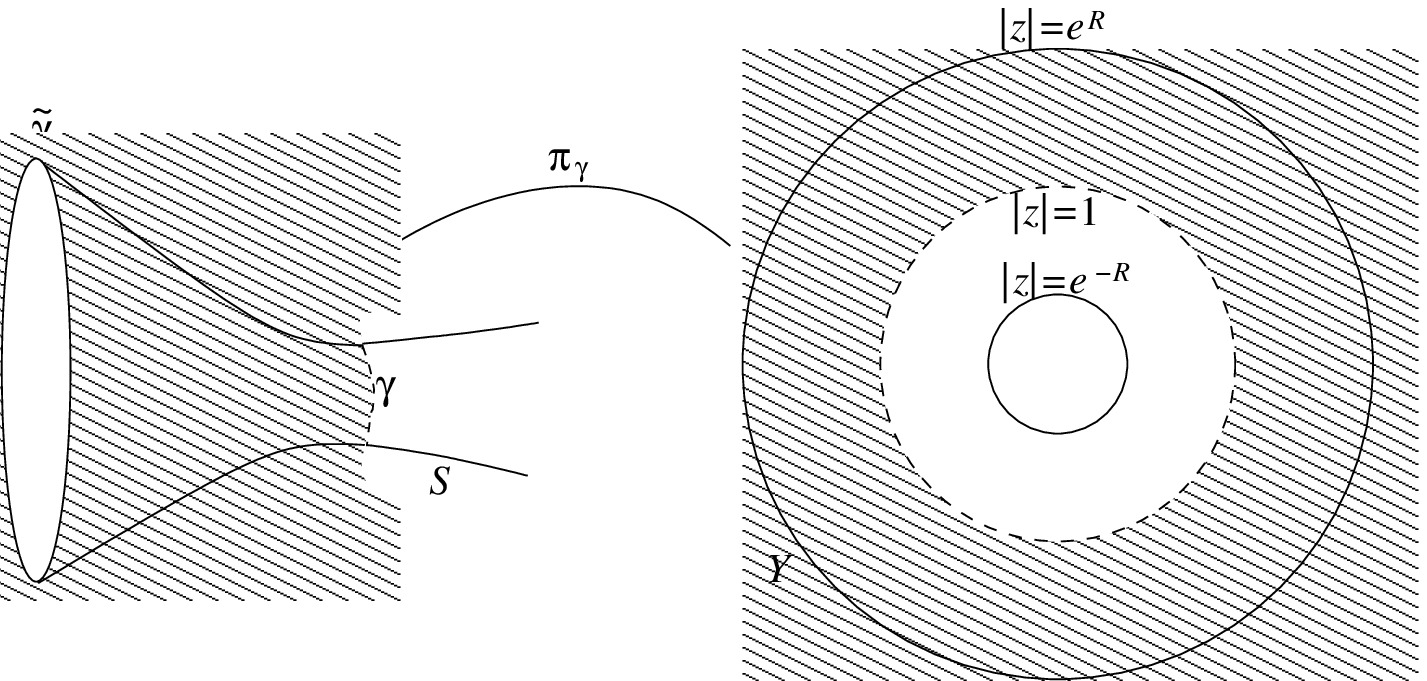}
\caption{Standard coordinates on the funnel}
\end{center}
\end{figure}

\subsection{The inhomogeneous Cauchy-Riemann equation on the surface} We want to
solve the inhomogeneous Cauchy-Riemann equation on $S$ with weighted uniform
estimates. In order to get good estimates it is useful to find functions $f\in
\H(S)$ with precise size control, i.e., $|f|\simeq e^\phi$ outside a
neighborhood of the zero set of $f$. With this function we can later modify an
integral formula to get a bounded solution to the $\dbar$-equation when the data
has compact support. The following Lemma provides such a function that in other
context has been termed a ``multiplier'':

\begin{lemma}~\label{multiplier}
Let $S$ be a finite Riemann surface and let $\phi$ be a subharmonic function
with bounded Laplacian. Then there is a function $f$ with hyperbolically
separated zero set $\Sigma$ such that $|f|\simeq  e^{\phi}$  whenever
$d(z,\Sigma)>\varepsilon$. Moreover if we fix any compact $K$ in $S$ it is
possible to find $f$ with the above properties and without zeros in $K$.
\end{lemma}
\begin{proof}
In any of the funnels $A_i$ we transfer the subharmonic weight $\phi$ to the
standard coordinate chart $1<|z|<e^{R_i}$. We define a weight $\phi_i$ on the
disk $|z|<e^{R_i}$ in such a way that $\phi_i$ has bounded invariant Laplacian 
and moreover $|\phi-\phi_i|<C$ on the region $1<|z|<e^{R_i}$. One way to do so
is the following: we assume from the very beginning that $\phi$ is smooth (this
is no restriction since otherwise it can be approximated by a smooth function).
Define 
\begin{equation}\label{eq1}
\phi_i(z)= \phi(z) \chi(z) + M_i \|z\|^2,
\end{equation}
where $\chi$ is a cutoff function such that $\chi\equiv 1$ in
$e^{R_i/2}<|z|<e^{R_i}$,  $\chi\equiv 0$ in $|z|<1$ and $M_i$ is taken big
enough such that $\phi_i$ is subharmonic and the invariant Laplacian of $\phi_i$
is bounded above and below.

We are under the hypothesis of the result from \cite{Seip95b} that states that
there is an holomorphic function in the disk $f_i$ with separated zero set
$Z(f_i)$ (in the hyperbolic metric of the disk) such that $|f_i|\simeq
e^{\phi_i}$ whenever $d(z,Z(f_i))>\varepsilon$. Since the hyperbolic metric of
the disk is comparable to the hyperbolic metric in the funnel, we have found a
function $f_i\in \H(A_i)$ with separated zero set such that $|f_i(z)|\simeq
e^{\phi(z)}$ if $d(z,Z(f_i))>\varepsilon$. Moreover dividing out $f_i$ by a
finite Blaschke product we can assume that $f_i$ is zero free in any prefixed
compact of the disk. 

We consider the ``core'' of $S$ to be $S\setminus \tilde A_i$, where $\widetilde
{A_i}$ are the outer part of the funnels mapped by $e^{S_i}<|z|<e^{R_i}$. The
values of the $S_i$ are taken so big as to make sure that the compact $K$ in the
hypothesis of the Lemma is contained in the core of $S$. We adjust the $f_i$
$i=1,\ldots,n$ as mentioned before to make sure that they are zero free in the
inner part of the funnels $1<|z|<e^{S_i}$. We finally define $f_0\equiv 1$ in
the core of $S$. 

To patch the different $f_i$ together we will need to solve a Cousin II problem
with bounds. Our data is $f_i$ defined on the inner parts of the funnels mapped
by $1<|z|<e^{S_i}$. The data are bounded above and below in the inner parts of
the funnels (because $\phi$ is bounded above and below in any compact of $S$
and $f_i$ have no zeros there). We want to find functions $g_i\in H(A_i)$ and
$g_0$ holomorphic on the core of $S$ such that $f_i=g_0/g_i$ in the inner part
of the funnel. If moreover $g_i$ and $g_0$ are bounded (above and below) then
the function $f$ defined as $f_ig_i$ in each of the funnels $A_i$ and $g_0$ on
the core of $S$ is holomorphic on $S$ and has the desired growth properties. To
find the functions $g_i$ observe that since the intersection of the funnel $A_i$
with the core of $S$ strictly separates the outer part of the funnel from the
inner part of the core we can reduce the Cousin II problem to solving a
$\dbar$-equation with bounded estimates of the solution on $S$ when the data is
bounded and with compact support (the support is in the inner part of the
funnels). This can be achieved by Lemma~\ref{dbaracotat}.
\end{proof}

With this function we can then obtain the following  result which is interesting
by itself:
\begin{theorem}\label{dbar}
Let $S$ be a finite Riemann surface and let $\phi$ be a subharmonic function
with a bounded  Laplacian. There is  a constant $C>0$ such that for any
$(0,1)$-form $\omega$ on $S$ there is a solution $u$ to the inhomogeneous
Cauchy-Riemann equation $\dbar u=\omega$ in $S$ with the estimate
\[
\sup_{z\in S} |u(z)|e^{-\phi(z)}\le C \sup_{z\in S} \langle\omega(z)\rangle
e^{-\phi(z)},\] 
whenever the right hand is finite. 
\end{theorem}
Recall that the notation $\langle\omega(z)\rangle$ means the hyperbolic norm of
$\omega$ at the point $z$.
\begin{proof}
Let  $w_i$ be the form $w$ restricted to the funnel $A_i$. We take a standard
coordinate chart and we may think of $w_i$ as a $(0,1)$-form defined on the disk
$|z|<e^{R_i}$ and with support in $1<|z|<e^{R_i}$. Consider as in the proof of
Lemma~\ref{multiplier} a subharmonic function $\phi_i$ in the disk with bounded
laplacian and such that $|\phi-\phi_i|<C$ if $1<|z|<e^{R_i}$.

By the results in \cite[Thm~2]{Ortega02} there is  a solution $u_i$ to the
problem $\dbar u_i =w_i$ in the disk $|z|<e^{R_i}$ with the  estimate
\[
\sup_{|z|<e^{R_i}} |u_i|e^{-\phi_i} \le C_i \sup_{1<|z|<e^{R_i}} 
\langle w_i\rangle e^{-\phi}
\] 
Observe that the hyperbolic metric of the disk and of the surface $S$ in the
funnel are equivalent. We consider $\tilde u_i = u_i \chi_i$, where  $\chi_i$ is
a cutoff function with support in $1<|z|<e^{R_i}$ and such that $\chi_i\equiv 1$
if $|z|>e^{R_i/2}$. The function $\tilde u_i$ is extended by $0$ to the
remaining of $S$ and it has the estimate $\sup_{S} |\tilde u_i|e^{-\phi}\le C_i
\sup_S \langle w\rangle e^{-\phi}$. Now $\dbar \tilde u_i$ coincides with $w$ on
the outer part of the funnel $A_i$. Thus the $(0,1)$-form $w_k=w-\sum_i \dbar
\tilde u_i$ has compact support in $S$ and it satisfies $\sup_S \langle
w_k\rangle e^{-\phi}\le \sup_S \langle w\rangle e^{-\phi}$. The desired solution
is then $u=\sum \tilde u_i +v$, where $v$ is such that $\dbar v = w_k$. We must
then solve $\dbar v = w_k$ with weighted uniform estimates but with the
advantage that $w_k$ has compact support $K$. 

Let $T(\omega_k)$ be a solution operator for $\partial\bar u =\omega_k$. We take
the operator $T$ given by Lemma~\ref{dbaracotat} the estimate $\sup_S
|T[w_k](z)|\le C_K \sup_K \langle w_k\rangle$ holds. Take $f$ with $|f|\simeq
e^{\phi}$ and without zeros in $K$ as given in Lemma~\ref{multiplier}. Then we
define $R$ as
\begin{equation}\label{pes}
R[\omega_k](z) = f(z) T[\omega_k/f](z),
\end{equation}
It solves $\bar\partial R[\omega_k] =\omega_k$ with the estimate
\[
\sup_S |R[\omega_k]| e^{-\phi} \le C_K \sup_K \langle\omega_k\rangle e^{-\phi}.
\]
The solution is thus $v=R[w_k]$.
\end{proof}

\section{The main results}

\begin{proposition}\label{interpolation}
A separated sequence $\Lambda\subset S$ is interpolating for $A_\phi(S)$ if and
only if the sequences $\Lambda_i=\Lambda \cap A_i$ are interpolating in
$A_\phi(A_i)$.
\end{proposition}

\begin{proof}
We only need to prove that we can pass from the local to the global
interpolation property. We split the proof in two steps
\begin{enumerate} 
\item From a funnel $A_i$ to global $S$: We need to prove that there are finite
sets $F_i\subset \Lambda_i$ such that $\cup_{i=1}^n (\Lambda_i\setminus F_i)$ is
interpolating globally.

\item Filling up the remainder. We shall prove that by adding a finite number of
points to an interpolating sequence we still get an interpolating sequence. Thus
$\Lambda$ is interpolating if $(\Lambda_1\setminus
F_1)\cup\cdots\cup(\Lambda_n\setminus F_n)$ is interpolating.
\end{enumerate}
\end{proof}

Let $\tilde \gamma$ be one of the closed curves on the boundary. Take a funnel
$A$  with outer end curve in $\tilde\gamma$ and inner end curve in $\gamma$. The
constant of interpolation in the funnel $A$ is $K>0$. Take a cutoff function
$\chi_\varepsilon$ with support in the funnel such that
$\langle\dbar\chi_\varepsilon\rangle <\varepsilon/(KC)$ (where $C$ is the
constant in Theorem~\ref{dbar}), the support is in a thick annulus of hyperbolic
thickness $M=M(\varepsilon,K,C)$. We consider a smaller funnel where
$\chi_\varepsilon\equiv 1$. The sequence $\Lambda$ in this smaller funnel has
still at most interpolation constant $K$. We can interpolate arbitrary
values on $\Lambda$ being small near the inner curve $\gamma$ of $A$ in the
following way. Take some values $v_\lambda$ with norm one. Take a function in
the funnel $f$ with norm at most $K$ that solves the interpolation problem. We
are going to approximate it by a function in $A$ that is small near $\gamma$.
Cut it off by $\chi_\varepsilon$ and correct via the following
inhomogeneous Cauchy-Riemann equation:

\[
\dbar u = f \dbar \chi
\]
The function $h=u-f\chi$ is holomorphic. By using Theorem~\ref{dbar} on
it is possible to solve the equation with a solution $u$ such
that $\sup|u|e^{-\phi}\le \varepsilon$. The function $h$ does not solve the
problem
directly but it almost does. We reiterate the procedure (interpolating the
error $v_\lambda-h(\lambda)$ and with a convergent series we get finally a
function $g$ such that $h(\lambda)=v_\lambda$, $\sup_A |h|e^{-\phi}\le
2$ and moreover in the inner half of the funnel that we denote by
$\tilde A$, $\sup_{\tilde A}|h|e^{-\phi}\le \varepsilon$. 

Now it is easier to make it global. Take a new cutoff function $\chi$ with
support in the funnel $A$ and that is one on the outer part of (i.e.
$A\setminus \tilde A$. Then we need to solve 
\[
\dbar u = h\dbar \chi,
\]
with good global estimates in $S$. These are  given by Theorem~\ref{dbar}. We
have solved the interpolation problem when the sequence lies in the funnels. For
the general situation we only need to add a finite number of points. The
existence of ``Blaschke''-type factors $h_\lambda(z)$ provided by
Theorem~\ref{blashcke} shows that $\Lambda\cup\lambda$ is interpolating if
$\Lambda$ is interpolating (it is immediate to build functions in the space such
that $f|_\Lambda\equiv 0$ and $f(\lambda)\ne 0$). \qed

For the sampling part we need the following definition
\begin{definition}\label{associats} Given the pair $(S,\phi)$ of a finite
Riemann surface and a subharmonic function defined on it, we associate to it the
pairs: $(D_i,\phi_i)_{i=1,\ldots n}$ of disks $D_i$ and subharmonic functions
$\phi_i$ defined on the disks as follow: If $A_i=\{1<|z|<e^{R^i}\}$,
$i=1,\ldots,n$ are the standard charts of the funnels of $S$ we define
$D_i=\{|z|<e^{R_i}\}$ and $\phi_i$ is any subharmonic function in $D_i$ 
such that $|\phi_i-\phi|<C$ in the region $1<|z|<e^{R_i}$, $\Delta \phi_i=\Delta
\phi$ in $e^{R_i/2}<|z|<e^{R_i}$ and $\Delta\phi_i \simeq 1$ in $|z|<e^{R_i/2}$.
They can be defined similarly as in \eqref{eq1}, but to make sure $\Delta
\phi_i=\Delta \phi$ we may take instead \[
\phi_i(z)=\phi(z)\chi(z)+M_i\psi(z),
\]
where $\psi$ is any bounded subharmonic function in $D_i$ 
such that $\Delta\psi(z)=1$ if $|z|<e^{R_i/2}$ and $0$ elsewhere.
\end{definition}
The funnels $A_i$ can be considered funnels of $S$ and they are subdomains of
$D_i$ too. We will exploit this double nature in the following theorem
\begin{theorem}\label{sampling}
Let $S$ be a finite Riemann surface and let $\phi$ be a subharmonic function
with bounded  Laplacian. A separated sequence $\Lambda$ is sampling for
$A_\phi(S)$ if and only if all the sequences in the funnels
$\Lambda_i=\Lambda_i\cap A_i\subset D_i$ are sampling sequences for
$A_{\phi_i}(D_i)$, where $(D_i,\phi_i)$ are the associated pairs to $S$ given by
Definition~\ref{associats}. 
\end{theorem}
Thus this Theorem and Proposition~\ref{interpolation}  show that the properties
of sampling and interpolation only depend on the behavior of the sequence and
the weight near the boundary pieces. 

To prove Theorem~\ref{sampling} we need some previous results
\begin{lemma}\label{uniqueness}
Let $S$ be a finite Riemann surface and let $\phi$ be a subharmonic function
with bounded  Laplacian. A  sequence $\Lambda\subset S$ is a uniqueness sequence
for $A_\phi(S)$ if and only if all the sequences in the funnels
$\Lambda_i=\Lambda_i\cap A_i\subset D_i$ are uniqueness sequences for
$A_{\phi_i}(D_i)$, where $(D_i,\phi_i)$ are the associated pairs to $S$ given by
Definition~\ref{associats}. 
\end{lemma}
\begin{proof}
It is easier to deal by negation. Let $\Lambda$ be contained in the zero set of
a function $f\in A_\phi(S)$. Therefore $\Lambda_i$ is in the zero set of $f\in
A_\phi(A_i)$. We divide by  a finite number of zeros $E_i$ and we obtain a new
function $g\in A_\phi(A_i)$ without zeros in $1<|z|\le e^{R_i/2}$ and such that
$\Lambda_i\setminus E_i\subset Z(g)$. Take the disk $D_i$ and consider the cover
by two open sets $|z|>1$ and $|z|<e^{R_i/2}$.  On the first set we have the
function $g$ and on the second the function $1$. The quotient is bounded above
and below in the intersection of the sets. This defines a bounded Cousin II in
the disk $D_i$
problem that can be solved with bounded data. We get a new function $h\in
A_{\phi_i}(D_i)$ that vanishes in $Z(g)$. We can now add the finite number of
zeros $E_i$ without harm. The reciprocal implication follows with the same
argument.
\end{proof}
The next result is inspired by a result of Beurling (\cite[pp.
351--365]{Beurling89b}) that relates the property of sampling sequence to that
of uniqueness for all weak limits of the sequence. In the context of the
Bernstein space (in the original work by Beurling) the space was fixed (it was
$\C$, the space of functions was fixed, the Bernstein class, and he considered
translates and limits of it of the sampling sequence). Here we need to move and
take limits of the sequence (by zooming on appropriate portions of it) but we
also need to change the support space (portions of $S$ near the funnel that look
like the unit disk) and we will also move the space of functions by changing the
weights. We need some definitions: 
\begin{definition}
We consider triplets $(D_n,\phi_n,\Lambda_n)$ where $D_n$ are disks
$D_n=D(0,r_n)\subset \D$, $\phi_n$ are subharmonic functions defined in a 
neighborhood of $D_n$  and $\Lambda_n$ is a finite collection of
points in $D_n$. We say that $(D_n,\phi_n,\Lambda_n)$ converges
 weakly to $(\D,\phi,\Lambda)$ (where $\D$ is the unit disk, $\phi$ a
subharmonic function in $\D$ and $\Lambda$ a discrete sequence in $\D$) if the
following conditions are fullfilled:
\begin{itemize}
\item The domains $D_n$ tend to $\D$, i.e.: $r_n\to 1$,
\item The weights $\phi_n$ tend to the weight $\phi$ in the sense
that $\Delta\phi_n$ as measures converges weakly to $\Delta\phi$.
\item The sequences $\Lambda_n$ converge weakly to $\Lambda$, i.e,
the measure $\sum_{\lambda\in\Lambda_n} \delta_{\lambda}$
converges weakly to the measure 
$\sum_{\lambda\in\Lambda} \delta_{\lambda}$.
\end{itemize}
\end{definition}
Let us fix a point $p\in S$. If a sequence of points $z_n\in S$ goes to
$\infty$, i.e. $d(z_n, p)\to\infty$, from a point $n_0$ on it will eventually
belong to the union of the funnels $A_1\cup \cdots \cup A_n$. If we take the set
of points $D_n=\{z\in S;\ d(z,z_n)<d(z_n,p)/2\}$ then $D_n$ is an hyperbolic
disk contained in the funnels if $n$ is big enough. In each of the $D_n$ we
consider the function $\phi_n=\phi|_{D_n}$ and $\Lambda_n=\Lambda\cap D_n$. Thus
for any sequence of points $z_n$ with $d(p,z_n)\to\infty$ we build a triplet
$(D_n,\phi_n,\Lambda_n)$ for $n$ big enough. 
\begin{definition}
Let $W(S,\phi,\Lambda)$ be the set of all triplets $(\D,\phi^*,\Lambda^*)$ which
are
weak limits of triplets $(D_n,\phi_n,\Lambda_n)$ associated to any sequence
$z_n$ such that $d(p,z_n)\to\infty$.
\end{definition}
The theorem of Beurling on our context is 
\begin{theorem}\label{beurling}
Let $S$ be Riemann surface of finite type and let $\phi$ be  a
subharmonic function with bounded  Laplacian. A separated sequence
$\Lambda$ is sampling for $A_\phi(S)$ if and only if 
\begin{itemize}
\item The sequence $\Lambda$ is a uniqueness set for $A_\phi(S)$
\item For any triplet
$(\D,\phi^*,\Lambda^*)\in W(S,\phi,\Lambda)$, the sequence $\Lambda^*$ is a
uniqueness
set for $A_{\phi^*}(\D)$.
\end{itemize}
\end{theorem}
\begin{proof} 
Let us prove that the uniqueness conditions imply that $\Lambda$ is a sampling
sequence. If it were not, there would be a sequence of functions $f_n\in
A_\phi(S)$ such that $\sup_{\Lambda}|f_n|e^{-\phi}\le 1/n$ and
$\sup_{S}|f_n|e^{-\phi}=1$. Take a sequence of points $z_n$ with
$|f_n(z_n)|e^{-\phi}\ge 1/2$. If $z_n$ are bounded we can take a subsequence of
points that we still denote $z_n$ convergent to $z^*\in S$ and by a normal
family argument there is a partial of $f_n$ convergent to $f\in A_\phi$, such
that $f|_\Lambda\equiv 0$, $f(z^*)\ne 0$ and this is not possible. Thus $z_n$
must be unbounded. Then we take the triplets $(D_n,\phi_n,\Lambda_n)$ associated
to $z_n$ and $D_n\to \D$ because $z_n\to\infty$ and the hyperbolic radius of
$D_n$ is $d(z_n,p)/2$. Since $\phi_n$ has bounded Laplacian, the mass of
$\Delta\phi_n$ restricted to any compact $K$ in $\D$ is bounded, thus we can
take a subsequence that converges weakly to a positive measure $\mu$ in $\D$
which satisfies $(1-|z|)^2\mu \simeq 1$ because all the mesures $\Delta\phi_n$
satisfy this inequalities with uniform constants. Let $\phi$ be such that
$\Delta\phi^*=\mu$. Since $\Lambda_n=D_n\cap \Lambda$ are all separated with
uniform bound, there is a weak limit $\Lambda^*$. The functions $f_n$ in the
disks can be modified by a factor $e^{g_n}$ in such a way that $h_n=f_ne^{g_n}$
satisfies $h_n(0)=1$ and $|h_n|\le e^{\phi_n+\Re(g_n)}$, if $n$ big enough and
$\sup_{\Lambda_n} |h_n| e^{-\phi_n+\Re(g_n)}\le 1/n$. We can add an harmonic
function $v$ to $\phi^*$ in such a way that $\phi_n+\Re(g_n)\to v+\phi^*$
uniformly  on compact sets. Thus $h_n$ has a partial convergent to $h\in
A_{\phi^*}$, $h(0)=1$ and $h|_{\Lambda^*}\equiv 0$ which was not possible by
assumption.

In the other direction, we assume that $\Lambda$ is a sampling sequence for
$A_\phi(S)$,  and $(\D,\phi^*,\Lambda^*)\in W(S,\phi,\Lambda)$. We want
to prove
that any $f\in A_\phi^*(\D)$ that vanishes in $\Lambda^*$ is identically $0$.
Take a sequence of points $z_n$ that escapes to infinity and 
$(D_n,\phi_n,\Lambda_n)$ the associated triple that converges weakly to
$(\D,\phi^*,\Lambda^*)$. As $\phi_n\to \phi^*$ and $\Lambda_n\to\Lambda^*$
uniformly on compact sets we can take a sequence of radii $s_n$ such that
$d(\Lambda\cap D(z_n,s_n),\Lambda^ *\cap D(0,s_n))<1/n$, $D(z_n,s_n)\subset
D(z_n,r_n)$ and $|\phi_n-\phi^*|\le 1/n$. If $f$ vanishes in $\Lambda^*$ that
means that $f$ is very small in $D(z_n,s_n)\cap \Lambda$. Assume that $f(0)=1$.
Take a cutoff function $\chi_n$ such that $\chi_n\equiv 0$ outside $D(z_n,s_n)$,
$\chi(z_n)=1$, and $\langle d\chi\rangle<\varepsilon_n$. Define
$g_n=f\chi_n-u_n$, where $\dbar u= f \dbar \chi_n$ is the solution estimates by
Theorem~\ref{dbar}. Clearly $g_n$ is small in all points of $\Sigma$ and it has
at least norm $1$. Thus we are contradicting the fact that $\Lambda$ is
sampling.

\end{proof}
Observe that one particular instance of finite Riemann surface, where we can
apply the result are the disks $D_i$ associated to the funnels with the metric
$\phi_i$. The final piece for the proof of Theorem~\ref{sampling} is then
\begin{lemma}\label{disc} If $S$ is  a finite Riemann surface, $\phi$ a
subharmonic function with bounded Laplacian and $\Lambda$ is a
uniformly separated sequence, then all possible weak limits coincide with the
weak limits of the disks associated to the surface, i.e,
\[
W(S,\phi,\Lambda)=W(D_1,\phi_1,\Lambda_1)\cup\cdots \cup
W(D_n,\phi_n,\Lambda_n).
\]
\end{lemma}
\begin{proof}
The proof amounts to the observation that the metric in $D_i$ converges
uniformly to the metric in $S$ as $z\to\partial \D_i$, and in the definition of
weak limits we only consider uniform convergence over compacts.
\end{proof}
Theorem~\ref{sampling} follows now immediately from
Theorem~\ref{beurling} and Lemmas \ref{uniqueness} and \ref{disc}. \qed

Now Theorem~\ref{sampling} and Theorem~\ref{interpolation} show that the
property of being a sampling/interpolating sequence are determined by the
behavior near the boundary, more precisely in the associated disks. In these
disks there is  a  precise description of the interpolating and sampling
sequences (see \cite{BerOrt95} and \cite{OrtSei98}) that can be transported to
the surface. If we rewrite it we get the density conditions of
Theorem~\ref{main}, but the disks are not hyperbolic disks on the surface, they
correspond to hyperbolic disks in disks $D_i$, but since the condition is only
relevant near the boundary, then the disks in both metrics look more an more
similar. Moreover the difference between the corresponding Green functions
converge to $0$ uniformly as we go to the boundary. Finally, as the sequence is
uniformly discrete and the Laplacian of the weight is bounded above and bellow,
the small difference is absorbed by the fact that the inequalities are strict
and this proves Theorem~\ref{main}. In fact it is possible to replace in the
definition of the density, \eqref{densities} the Green function $g_r$ of
$D(z,r)$ by the Green function $g$ of $S$, because as before $\sup_{w\in
D(z,r)}|g_r(z,w)-g(z,w)|\to 0$  as $z$ approaches the boundary.

\section{Some $L^p$-variants}\label{Lebesgue}
We have considered up to now pointwise growth restrictions. It is possible to
obtain from our Theorem other results in different Banach spaces of holomorphic
functions. Consider for instance the weighted Bergman spaces 
\[
A_\phi^p(S) = \{f\in \mathcal H (S);\ \int_S |f|^p e^{-\phi}\, dA<+\infty\},
\]
where $dA$ is s the hyperbolic area measure in $S$ and $p\in [1,\infty)$. The
natural problem in this context is the following: 
\begin{definition} Let $S$ be a finite Riemann surface, and let $\phi$ be a
subharmonic function with bounded Laplacian bigger than one, i.e.,
$1+\varepsilon <\Delta \phi <M$. 
\begin{itemize} 
\item A sequence $\Lambda\subset S$ is interpolating for $A_\phi^p(S)$ if for
any values $v_\lambda$ such that 
\[
\sum_{\lambda\in\Lambda} |v_\lambda|^p e^{-\phi(\lambda)}  <\infty
\] 
there is a function $f\in A_\phi^p(S)$ such that $f(\lambda)=v_\lambda$.

\end{itemize}
\end{definition}
The spaces $A^p_\phi$ can be empty if we only ask $\phi$ to be with positive
bounded Laplacian. It is then natural to require that the Laplacian is strictly
bigger than one so that the Laplacian plus the curvature of the metric in the
manifold is strictly positive and there are functions in the space (consider the
case of the disk $S=\D$ for instance).


Let $\phi_0$ be a subharmonic function in $S$ such that $\Delta \phi_0=1$. The
corresponding theorem will be 
\begin{theorem}\label{samplinglp} Let $S$ be a finite Riemann surface, and let
$\phi$ be a subharmonic function with bounded Laplacian strictly bigger than
one. Let $p\in [1,+\infty)$ and $\Lambda$ be a separated sequence. 
\begin{itemize}
\item The sequence $\Lambda$ is interpolating for $A_\phi^p(S)$ if and only if
$D^+_{(\phi-\phi_0)}(\Lambda)<1/p$.
\end{itemize}
\end{theorem}
In the case of the unit disk $dA(z)=(1-|z|)^{-2}$ this description is
well-known, see  for instance \cite[Thm 2,3]{Seip98}.
\begin{proof}
The proof of the theorem is the same mutatis-mutandi as in the $L^\infty$
setting. The basic tool that allows us to glue the pieces together is the next
theorem which is the generalization of Theorem~\ref{dbar} and it is proved in
the same way:

\begin{theorem}\label{dbarlp}
Let $S$ be a finite Riemann surface, let $\phi$ be a subharmonic function with a
bounded  Laplacian strictly bigger than one and let $p\in [1,\infty)$. There is 
a constant $C=C(p,S)>0$ such that for any $(0,1)$-form $\omega$ on $S$ there is
a solution $u$ to the inhomogeneous Cauchy-Riemann equation $\dbar u=\omega$ in
$S$ with the estimate 
\[
\int_{S} |u(z)|^p e^{-\phi(z)}dA(z)\le C \int_{S} \langle\omega(z)\rangle^p
e^{-\phi(z)} dA(z),
\] 
whenever the right hand is finite. 
\end{theorem}
The proof of this result is again the same as in Theorem~\ref{dbar}. We can 
separately solve the C-R equation in each funnel  using Theorem~2 from
\cite{Ortega02}.  We glue them together with a C-R equation with data that has
compact support that can be solved with the operator \eqref{pes}. 
%
%
\end{proof}

\def\cprime{$'$}
\providecommand{\bysame}{\leavevmode\hbox to3em{\hrulefill}\thinspace}
\providecommand{\MR}{\relax\ifhmode\unskip\space\fi MR }
\providecommand{\MRhref}[2]{%
  \href{http://www.ams.org/mathscinet-getitem?mr=#1}{#2}
}
\providecommand{\href}[2]{#2}

\end{document}